\documentclass{amsart}

\usepackage{amssymb}
\usepackage{microtype}
\usepackage{fullpage}
\usepackage{hyperref}

\hypersetup{colorlinks=true, citecolor=blue, linkcolor=blue, urlcolor=blue, pdfstartview=FitH, pdfauthor=Valentin Blomer and Gergely
Harcos and Péter Maga, pdftitle=On the global sup-norm of GL(3) cusp forms}

\renewcommand{\leq}{\leqslant}
\renewcommand{\geq}{\geqslant}

\newcommand{\ZZ}{\mathbb{Z}}
\newcommand{\QQ}{\mathbb{Q}}
\newcommand{\RR}{\mathbb{R}}
\newcommand{\CC}{\mathbb{C}}
\newcommand{\NN}{\mathbb{N}}

\newcommand{\PGL}{\mathrm{PGL}}
\newcommand{\GL}{\mathrm{GL}}
\newcommand{\SL}{\mathrm{SL}}
\renewcommand{\O}{\mathrm{O}}
\newcommand{\SO}{\mathrm{SO}}
\newcommand{\Z}{\mathrm{Z}}
\newcommand{\U}{\mathrm{U}}
\newcommand{\M}{\mathrm{M}}

\newcommand{\Wc}{\mathcal{W}}
\newcommand{\Hc}{\mathcal{H}}
\newcommand{\eps}{\varepsilon}
\newcommand{\bs}{\backslash}
\newcommand{\id}{1_n}
\newcommand{\ov}[1]{\overline{#1}}

\DeclareMathOperator{\Ad}{Ad}
\DeclareMathOperator{\sign}{sign}
\DeclareMathOperator{\diag}{diag}

\theoremstyle{plain}
\newtheorem{theorem}{Theorem}
\newtheorem{lemma}{Lemma}

\theoremstyle{remark}
\newtheorem*{remarks}{Remarks}

\theoremstyle{definition}
\newtheorem*{acknowledgement}{Acknowledgements}

\begin{document}

\author{Valentin Blomer}
\author{Gergely Harcos}
\author{P\'eter Maga}

\address{Mathematisches Institut, Bunsenstr. 3-5, D-37073 G\"ottingen, Germany}\email{vblomer@math.uni-goettingen.de}
\address{Alfr\'ed R\'enyi Institute of Mathematics, Hungarian Academy of Sciences, POB 127, Budapest H-1364, Hungary}\email{gharcos@renyi.hu}
\address{Central European University, Nador u. 9, Budapest H-1051, Hungary}\email{harcosg@ceu.edu}
\address{Alfr\'ed R\'enyi Institute of Mathematics, Hungarian Academy of Sciences, POB 127, Budapest H-1364, Hungary}\email{magapeter@gmail.com}

\title{On the global sup-norm of $\GL(3)$ cusp forms}

\thanks{First author partially supported by the DFG-SNF lead agency program grant BL~915/2-1. Second and third author supported by NKFIH (National Research, Development and Innovation Office) grants NK~104183, ERC\underline{\phantom{ }}HU\underline{\phantom{ }}15~118946, K~119528. Second author also supported by ERC grant AdG-321104, and third author also supported by the Postdoctoral Fellowship of the Hungarian Academy of Sciences.}


\begin{abstract}
Let $\phi$ be a spherical Hecke--Maa{\ss} cusp form on the non-compact space $\PGL_3(\ZZ)\bs\PGL_3(\RR)$. We establish various pointwise upper bounds for $\phi$ in terms of its Laplace eigenvalue $\lambda_\phi$. These imply, for $\phi$ arithmetically normalized and tempered at the archimedean place, the bound
\[\|\phi\|_\infty\ll_\eps \lambda_{\phi}^{39/40+\eps}\]
for the global sup-norm (without restriction to a compact subset). On the way, we derive a new uniform upper bound for the $\GL_3$ Jacquet--Whittaker function.
\end{abstract}


\maketitle

\section{Introduction}

Eigenfunctions of the Laplace--Beltrami operator $\Delta$ on a Riemannian manifold $X$ are the spectral building blocks of the Hilbert space $L^2(X)$. It is therefore a classical question  to study their analytic properties asymptotically as the eigenvalue tends to infinity.  In addition to its intrinsic interest in global analysis, a relation to quantum physics is provided by the fact that $e^{-it\Delta}$ is the time evolution operator of the Schr\"odinger equation describing a freely moving particle on $X$. Therefore, an $L^2$-eigenfunction of $\Delta$ is understood in quantum physics as a bound state, and its absolute square is interpreted as the probability density of the corresponding stationary wave. Somewhat unexpectedly, number theory also enters the scene, namely when the manifold $X$ possesses some additional arithmetic structure such as a commutative family of arithmetically defined Hecke operators (which are normal and commute with $\Delta$ as well). In this case, the most interesting Laplace eigenfunctions are those that are in addition eigenfunctions of the Hecke algebra, and as such are amenable to number theoretic tools. We will present an example of this kind in a moment.

A fundamental quantity associated to an $L^2$-normalized eigenfunction $\phi$ with eigenvalue $\lambda_{\phi}$ is its sup-norm. A good upper bound for $\| \phi \|_{\infty}$ can be seen as a basic measure of equidistribution of its mass on $X$. If $X$ is compact (or $X$ is non-compact but $\phi$ is restricted to a fixed compact subset $\Omega\subseteq X$), then we have the general bound proved by H\"ormander~\cite{Ho}
\begin{equation}\label{laplace1}
\|\phi\|_{\infty}\ll \lambda^{(d-1)/4},\qquad d=\dim X.
\end{equation}
This bound is sharp, e.g. it is attained for the $d$-sphere $S^d$ for any $d\geq 1$ and for special eigenfunctions $\phi$. If $X$ is a compact locally symmetric space of dimension $d$ and rank $r$ (or we restrict to a compact subset $\Omega$ of such a space) and we require that $\phi$ is not only a Laplace eigenfunction but an eigenfunction of all differential operators invariant under the group of isometries of the universal cover of $X$, then we have the stronger bound of Sarnak~\cite{Sa}
\begin{equation}\label{laplace2}
\|\phi\|_{\infty}\ll \lambda^{(d-r)/4}.
\end{equation}
Even though neither \eqref{laplace1} nor \eqref{laplace2} are conjectured to be sharp for negatively curved manifolds, they provide a robust framework to work with.

The situation changes completely for the global sup-norm on non-compact manifolds, in which case \eqref{laplace1} and \eqref{laplace2} no longer need to be true. A typical example is a locally symmetric space $X = \Gamma \backslash G/K$, where $G$ is a non-compact semi-simple Lie group, $K\leq G$ is a maximal compact subgroup, and $\Gamma\leq G$ is a non-uniform lattice. It turns out that in this case the sup-norm of an eigenfunction $\phi$ is often determined by its behavior in the cuspidal regions of $X$, even though it may eventually decay very quickly. The simplest -- and the only thoroughly explored -- example is $G =\SL_2(\RR)$, $K =\SO_2(\RR)$, $\Gamma=\SL_2(\ZZ)$, so that $X = \Gamma\backslash G/K = \Gamma \backslash \mathcal{H}_2$ (where $\mathcal{H}_2 = \{x + iy:\ y > 0\}$) is the familiar modular surface. It is an \emph{arithmetic manifold} in the above sense, as it admits the standard family of Hecke operators. A joint $L^2$-eigenfunction of the Laplace and the Hecke operators is often called a \emph{Hecke--Maa{\ss} cusp form}. It decays exponentially  as $y \rightarrow \infty$, and it is conjectured to satisfy $\|\phi|_{\Omega}\|_\infty\ll_{\eps,\Omega} \lambda_{\phi}^{\eps}$ as $\lambda_{\phi} \rightarrow \infty$ for every compact subset $\Omega \subseteq X$ and every $\eps > 0$, but in the cuspidal region it has considerable size. Precisely, we have $\phi(i\lambda_\phi^{1/2}/(2\pi))=\lambda_\phi^{1/12+o(1)}$, cf.\ \cite{Sa}, reflecting the very similar behavior of the normalized $\GL_2$ Whittaker function
\begin{equation}\label{gl2whittaker}
W_{\nu}(x):=\frac{\sqrt{x}K_{\nu}(2\pi x)}{|\Gamma(1/2+\nu)|},\qquad\nu\in i\RR.
\end{equation}
Indeed, $W_\nu(x)$ decays exponentially as $x\rightarrow \infty$, but it has a large bump: $W_{it}(t/(2\pi)) \asymp t^{1/6}$ (see \cite{Ba} for a uniform asymptotic expansion). \\

A systematic study of this behavior for Hecke--Maa{\ss} cusp forms on the locally symmetric space
\[X=X_n=\GL_n(\ZZ)\Z(\RR)\bs\GL_n(\RR)/\O_n(\RR)\]
and its congruence covers was initiated by Brumley and Templier~\cite{BT}.
Here, $\Z(\RR)$ is the center of $\GL_n(\RR)$ and $\O_n(\RR)$ is the orthogonal subgroup.
In particular, it was proved in \cite{BT} that \eqref{laplace2} fails for $X=X_n$ when $n \geq 5$, namely the \emph{global} sup-norm on $X$ is significantly larger than the \emph{local} sup-norm on a fixed compact subset $\Omega \subseteq X$. Precisely, the lower bound by Brumley--Templier~\cite{BT} for the global sup-norm and the upper bound by Blomer--Maga~\cite{BM} for the local sup-norm can be contrasted as
\[\|\phi\|_{\infty}\gg_\eps\lambda_\phi^{n(n-1)(n-2)/24-\eps}>\lambda_\phi^{n(n-1)/8-\delta_n}\gg_{\Omega} \|\phi|_{\Omega}\|_\infty,\]
where $\delta_n>0$ is a constant depending only on $n$. In addition, in the case $n=3$, Brumley and Templier derived the upper bound \cite[Prop.~1.6]{BT}
\begin{equation}\label{BrumleyTemplier}
\| \phi \|_{\infty}  \ll \lambda_{\phi}^{5/2}\qquad\text{on}\qquad X=X_3,
\end{equation}
by using the rapid decay of $\phi$ high in the cusp and making the dependence of \eqref{laplace2} on the injectivity radius in the remaining piece of the manifold explicit. We note in passing that in the case $n=3$, it is known that any savings $\delta_3<1/124$ is admissible\footnote{By saying this we do not mean to imply that $1/124$ is a special threshold beyond which the bound fails.} for the upper bound of the local sup-norm \cite{HRR}, while in the case $n=2$, the global sup-norm problem has been studied extensively (see \cite{IS,Sah1,Sah2,BHMM}
and the references therein).

Despite these important advances, the investigation of the global sup-norm of eigenfunctions on non-compact symmetric spaces of rank exceeding one
is still its infancy, and in this article we take a closer look at the rank two example $n=3$ with the aim of proving considerably
stronger bounds than \eqref{BrumleyTemplier} by a different technique.  With this in mind,
let $\phi$ be a Hecke--Maa{\ss} cusp form on $\GL_3$ over $\QQ$, which is spherical at every place and has trivial central character,
regarded as a complex-valued function on the quotient space $X=X_3$. Alternatively, we may think of $\phi$ as a complex-valued function on $\GL_3(\RR)$ satisfying
\begin{displaymath}
\phi(\gamma h g k) = \phi(g)\qquad\text{for all}\qquad\gamma\in\GL_3(\ZZ),\ h\in\Z(\RR),\ k\in\O_3(\RR),
\end{displaymath}
which is further an eigenfunction of the invariant differential operators and the Hecke operators \cite[Sections~6.1--6.4]{G}. We recall from \cite[Sections~1.2--1.3]{G} that the symmetric space $\Z(\RR)\bs \GL_3(\RR) / \O_3(\RR)$ can be represented by matrices of the generalized upper half-plane
\begin{displaymath}
\Hc_3:=\left\{z=\begin{pmatrix}1 & x_2 & x_3 \\ & 1 & x_1 \\ & & 1\end{pmatrix}
\begin{pmatrix}y_1y_2 & & \\ & y_1 & \\ & & 1\end{pmatrix}:\ y_1,y_2>0,\ x_1,x_2,x_3\in\RR\right\},
\end{displaymath}
and the quotient space $X=\GL_3(\ZZ)\bs\Hc_3$ has a fundamental domain lying in the Siegel set defined by
\begin{equation}\label{xybounds}
|x_1|,|x_2|,|x_3|\leq 1/2\qquad\text{and}\qquad y_1,y_2\geq\sqrt{3}/2.
\end{equation}

In this paper, we provide upper bounds for $|\phi(z)|$ in terms of the height parameters $y_1$, $y_2$ (assuming they are at least $\sqrt{3}/2$) and the Laplace eigenvalue $\lambda_{\phi}$, and examine what they yield for the global sup-norm
\[\|\phi\|_\infty=\sup_{z\in\Hc_3}|\phi(z)|.\]
We shall work with the following natural normalizations:
\begin{itemize}
\item $\phi$ is \emph{arithmetically normalized} if it has leading Fourier coefficient $\lambda_\phi(1,1)=1$ with respect to the standard Jacquet--Whittaker function (cf.\ \cite[Thm.~6.4.11]{G} and \eqref{eq:f-w-expansion} below);\\[-9pt]
\item $\phi$ is \emph{$L^2$-normalized} if it has $L^2$-norm $1$ with respect to the measure on $X$ induced by the standard left-invariant probability measure on $\Hc_3$ (cf.\ \cite[Prop.~1.5.3]{G}).
\end{itemize}
By \cite[Lemma~1]{B} and its proof, these two normalizations differ by a positive constant times $L(1,\phi,\Ad)^{1/2}$ when $\phi$ is tempered at the archimedean place, and this would also be true for non-tempered forms if we slightly renormalized the standard Jacquet--Whittaker function as in the display below \cite[(2.13)]{B} with the effect of correspondingly changing the notion of ``arithmetically normalized''. With these conventions, our main results are as follows (see also the remarks after the theorems):

\begin{theorem}\label{theorem_2} Let $\phi$ be an arithmetically normalized Hecke--Maa{\ss} cusp form on $X$. Assume that $\phi$ is tempered at the archimedean place. Then for any $z\in\Hc_3$ with $y_1,y_2\geq\sqrt{3}/2$, and for any $\eps>0$, we have
\begin{equation}\label{thm2bound}
\phi(z)\ll_\eps\min(y_1,y_2)\left(\frac{\lambda_{\phi}^{1+\eps}}{y_1y_2} + \frac{\lambda_{\phi}^{3/2+\eps}}{(y_1y_2)^2}\right).
\end{equation}
\end{theorem}

\begin{theorem}\label{theorem_3} Let $\phi$ be an $L^2$-normalized Hecke--Maa{\ss} cusp form on $X$. Then for any $z\in\Hc_3$ with $y_1,y_2\geq\sqrt{3}/2$, we have
\begin{equation}\label{thm3bound}\phi(z)\ll\lambda_{\phi}^{3/4} + \lambda_{\phi}^{5/8}y_1y_2.\end{equation}
\end{theorem}

\begin{theorem}\label{theorem_1} Let $\phi$ be an arithmetically normalized Hecke--Maa{\ss} cusp form on $X$. Assume that $\phi$ is tempered at the archimedean place. Then for any $\eps>0$, we have
\begin{equation}\label{thm1bound}\|\phi\|_\infty\ll_\eps \lambda_{\phi}^{39/40+\eps}.\end{equation}
\end{theorem}

\begin{remarks}\label{remarks}
Assume that $\phi$ is tempered at the archimedean place. If $\phi$ is arithmetically normalized, then \eqref{thm3bound} holds with the extra factor $L(1,\phi,\Ad)^{1/2}\ll_\eps\lambda_{\phi}^\eps$ on the right hand side by the work of Brumley~\cite[Cor.~2]{Br} or Li~\cite[Thm.~2]{L}. Similarly,
if $\phi$ is $L^2$-normalized, then \eqref{thm2bound} and \eqref{thm1bound} hold with the extra factor $L(1,\phi,\Ad)^{-1/2}\ll_\eps\lambda_{\phi}^{1/2+\eps}$ on the right hand side by the work of Brumley~\cite[Thm.~3]{Br2} (see also \cite[Appendix]{La}). If $\phi$ is self-dual, i.e.\ $\phi$ is the symmetric square of a classical (even or odd) Hecke--Maa{\ss} cusp form on $\SL_2(\ZZ)\bs\Hc_2$ (cf.\ \cite[Thm.~A]{R}), we even know that $L(1,\phi,\Ad)^{-1/2} \ll_\eps \lambda_{\phi}^\eps$ by a result of Ramakrishnan and Wang~\cite[Cor.~C]{RW}, so that in this case no adjustment to \eqref{thm2bound} and \eqref{thm1bound} is necessary. Finally, we note that for $\phi$ self-dual, the exponent $39/40$ in \eqref{thm1bound} cannot be lowered below $3/8$, as follows from the work of Brumley and Templier~\cite[Cor.~1.10]{BT}.
\end{remarks}

We prove our results by employing two very different methods. On the one hand, we estimate the Fourier--Whittaker expansion of $\phi$ termwise, which eventually leads to Theorem~\ref{theorem_2}. Unlike in the rank one case $n=2$, the fact that the group of upper-triangular unipotent $3\times 3$ matrices is not commutative leads to an additional sum over an infinite subset of $\SL_2(\ZZ)$ which requires careful treatment. As a result of independent interest and as a contribution to the analytic theory of higher rank Whittaker functions, we provide in Lemma~\ref{lemma:whittaker_estimate} a new uniform upper bound for the $\GL_3$ Jacquet--Whittaker function.

On the other hand, as in earlier approaches on compact spaces, we use a pre-trace formula, but we make the analysis of the geometric side uniform in the height of the considered point $z \in \mathcal{H}_3$; this is familiar for the group $\GL_2$, but seems to have not yet been worked out in higher rank. This leads to Theorem~\ref{theorem_3}. A key ingredient of the proof is Lemma~\ref{lemma:lower_bound_on_cartan_projection}, which states a simple but efficient bound for the norm of the Cartan projection of upper-triangular matrices. As we do not use amplification, the proof is independent of  Hecke operators and consequently the result also holds for all $L^2$-normalized Maa{\ss} forms, not just the Hecke eigenforms.

Our final Theorem~\ref{theorem_1} is a combination of Theorems~\ref{theorem_2} and \ref{theorem_3}. There is nothing particularly special about the exponent $39/40$, except that the corresponding bound \eqref{thm1bound} is considerably stronger than \eqref{BrumleyTemplier} and already fairly close to \eqref{laplace2}. Marginal improvements are possible, for instance by inserting an amplifier into the pre-trace formula.

\begin{acknowledgement} We thank Gerg\H o Nemes for valuable discussions concerning the special functions that appear in this paper.
We also thank the referee whose careful reading and detailed comments helped us to improve the exposition.
\end{acknowledgement}

\section{Archimedean parameters and the Weyl group}

Associated to every  Hecke--Maa{\ss} cusp form $\phi$ on $X$ (in fact, the Hecke property is not needed for this discussion), there is an $S_3$-orbit of (archimedean) Langlands parameters
\begin{align}\label{mutriple}
(\mu_1,\mu_2,\mu_3)\in\CC^3&\qquad\text{satisfying}\qquad\mu_1+\mu_2+\mu_3=0,\\
\intertext{and a corresponding $S_3$-orbit of (archimedean) spectral parameters}
\notag
(\nu_0,\nu_1,\nu_2)\in\CC^3&\qquad\text{satisfying}\qquad \nu_0=\nu_1+\nu_2.
\end{align}
We follow the conventions of \cite{Bu} except that we shift $\nu_1$ and $\nu_2$ there by $-1/3$, so that (cf.\ \cite[(8.3)]{Bu})
\[(\mu_1,\mu_2,\mu_3)=(\nu_1+2\nu_2,\nu_1-\nu_2,-2\nu_1-\nu_2),\]
\begin{equation}\label{laplaceeigenvalue}
\lambda_\phi=1-\frac{1}{2}\left(\mu_1^2+\mu_2^2+\mu_3^2\right)=1-3\nu_1^2-3\nu_1\nu_2-3\nu_2^2,
\end{equation}
and $(1/3+\nu_2,1/3+\nu_1)\in\CC^2$ are the spectral parameters in \cite[Section~6]{G}\footnote{Unfortunately, the relevant literature has many inconsistencies that are hard to track. For example, in Goldfeld's book \cite{G}, Definition~6.5.2 is inconsistent with Definition~9.4.3, and hence Theorem~6.5.15 contradicts Theorem~10.8.6 and the subsequent discussion. In fact $\nu_1$ and $\nu_2$ should be flipped in \cite[Thm.~6.5.15]{G}, in harmony with \cite[(8.3)]{Bu}. Compare also the functions $I_\nu$ on \cite[p.~19]{Bu} and \cite[p.~154]{G}.}. The Weyl group $S_3$ acts by permutations on the Langlands parameters, hence an $S_3$-orbit of spectral parameters takes the shape
\begin{equation}\label{orbit}
\bigl\{(\nu_0,\nu_1,\nu_2),\ (\nu_2,-\nu_1,\nu_0),\ (-\nu_1,\nu_2,-\nu_0),\ (-\nu_0,-\nu_2,-\nu_1),\ (-\nu_2,-\nu_0,\nu_1),\ (\nu_1,\nu_0,-\nu_2)\bigr\}.
\end{equation}

If $\phi$ is tempered at the archimedean place (as assumed in Theorems~\ref{theorem_2} and \ref{theorem_1}), the Langlands parameters and the spectral parameters are purely imaginary. Moreover, each $S_3$-orbit of spectral parameters contains a unique triple with nonnegative imaginary parts:
\begin{equation}\label{eq:nu-sum}
(\nu_0,\nu_1,\nu_2)\in(i\RR_{\geq 0})^3\qquad\text{satisfying}\qquad \nu_0=\nu_1+\nu_2.
\end{equation}
From now on we regard this triple as the spectral parameter of $\phi$. For convenience, we also introduce
\begin{equation}\label{T0}
T_0:=\max(2,|\nu_0|,|\nu_1|,|\nu_2|)=\max(2,|\nu_0|)\asymp\lambda_\phi^{1/2}.
\end{equation}

\section{The Fourier--Whittaker expansion}

The restriction of $\phi$ to $\Hc_3$ admits a Fourier--Whittaker expansion (cf.\ \cite[(6.2.1)]{G})
\begin{equation}\label{eq:f-w-expansion}
\phi(z)=\sum_{m_1=1}^{\infty} \sum_{m_2\neq 0} \frac{\lambda_{\phi}(m_1,m_2)}{|m_1m_2|} \sum_{\gamma\in\U_2(\ZZ)\bs \SL_2(\ZZ)} \Wc_{\nu_1,\nu_2}^{\sign(m_2)} \left( \begin{pmatrix}
|m_1m_2| & & \\ & m_1 & \\ & & 1
\end{pmatrix} \begin{pmatrix}
\gamma & \\ & 1
\end{pmatrix} z \right),
\end{equation}
where $\U_2$ is the subgroup of upper-triangular unipotent matrices $\bigl(\begin{smallmatrix}1 & * \\ & 1 \end{smallmatrix}\bigr)$ and $\Wc_{\nu_1,\nu_2}^{\pm}:\GL_3(\RR)\to\CC$ is the standard Jacquet--Whittaker function (cf.\ \cite[(6.1.2)]{G} and \cite[(2.10)--(2.11)]{B}). The function $\Wc_{\nu_1,\nu_2}^{\pm}$ is invariant under $\Z(\RR)$, right-invariant under $\O_3(\RR)$, and its restriction to $\Hc_3$ has the following integral representation by a result of Vinogradov--Takhtadzhyan \cite{VT} (cf.\ \cite[(2.12)]{B}):
\begin{multline}\label{eq:whittaker}
\left(\prod_{j=0}^2\frac{\Gamma\left(\frac{1}{2} + \frac{3}{2}\nu_j\right)}{\pi^{\frac{1}{2}+\frac{3}{2}\nu_j}}\right)
\Wc_{\nu_1,\nu_2}^{\pm}\left(\begin{pmatrix}
1 & x_2 & x_3 \\ & 1 & x_1 \\ & & 1
\end{pmatrix} \begin{pmatrix}
y_1y_2 & & \\ & y_1 & \\ & & 1
\end{pmatrix}\right) = \\[6pt]
4 e(x_1 \pm x_2) y_1^{1+\frac{\nu_1-\nu_2}{2}}y_2^{1+\frac{\nu_2-\nu_1}{2}}
\int_0^{\infty} K_{\frac{3}{2}\nu_0}(2\pi y_1\sqrt{1+u}) K_{\frac{3}{2}\nu_0}(2\pi y_2\sqrt{1+u^{-1}}) u^{\frac{3}{4}(\nu_1-\nu_2)} \frac{du}{u},
\end{multline}
where $K_\nu$ is the usual $K$-Bessel function. This formula is a special case of \cite[Thm.~2.1]{S} with the constant $2^{n-1}$ corrected to $2^{2n-3}$ there, as recorded by \cite[(4.3) \& p.~132]{S2}. Indeed, \eqref{eq:whittaker} follows readily from the first display of \cite[p.~318]{S}, by updating the constant $4$ to $8$, then substituting $u^{1/2}$ for $u$, and finally noting that $(\nu_1,\nu_2)$ and $(y_1,y_2)$ there equal $(1/3+\nu_1,1/3+\nu_2)$ and $(y_2,y_1)$ in our notation\footnote{Accordingly, in \cite[(6.1.3)]{G}, the constant $4$ should be $8$, while $\nu_1$ and $\nu_2$ should be flipped.}. An independent verification of \eqref{eq:whittaker} can be obtained by comparing carefully \cite[(5.8)]{BB2} with \cite[(10.1)]{Bu} (or more directly with \cite[(10.2)]{Bu}, after inserting a missing factor of $\pi^{-2}$ on the right hand side there).

We remark in passing that by \cite[Prop.~5.5.2, Eq.~(5.5.5), Thm.~5.9.8]{G}, the left hand side of \eqref{eq:whittaker} is the same over the whole spectral $S_3$-orbit \eqref{orbit}, hence in the tempered case (i.e.\ when $\nu_j\in i\RR$) the corresponding Fourier coefficients $\lambda_\phi(m_1,m_2)$ only differ from each other by six constants on the unit circle. At any rate, since the unipotent part (i.e.\ the entries $x_1,x_2,x_3$) acts on the right hand side of \eqref{eq:whittaker} via a rotation, it is convenient to introduce also
\begin{equation}\label{eq:whittaker_diagonal}
\begin{split}
\tilde{W}_{\nu_1,\nu_2}(y_1,y_2):= &\ 4\pi^{\frac{3}{2}} \prod_{j=0}^2\left| \Gamma\left(\frac{1}{2} + \frac{3}{2}\nu_j\right)\right|^{-1} y_1y_2 \left(\frac{y_1}{y_2}\right)^{\frac{\nu_1-\nu_2}{2}} \\ & \times \int_0^{\infty} K_{\frac{3}{2}\nu_0}(2\pi y_1\sqrt{1+u}) K_{\frac{3}{2}\nu_0}(2\pi y_2\sqrt{1+u^{-1}}) u^{\frac{3}{4}(\nu_1-\nu_2)} \frac{du}{u}.
\end{split}
\end{equation}

\section{An upper bound for the Jacquet--Whittaker function}

In this section we estimate, under the assumption \eqref{eq:nu-sum}, the function $\Wc_{\nu_1,\nu_2}^{\pm}$ given by \eqref{eq:whittaker}, or equivalently the function $\tilde{W}_{\nu_1,\nu_2}$ given by \eqref{eq:whittaker_diagonal}. We start with a bound involving the normalized $\GL_2$ Whittaker function introduced in \eqref{gl2whittaker}:

\begin{lemma}\label{lemma:GL(2)_whittaker_integral}
Assume that $\nu\in i\RR$, and put $T:=\max(1/2,|\nu|)$. Then, for any $t,A>0$, we have
\begin{displaymath}
\int_t^{\infty} \frac{W_{\nu}^2(x)}{\sqrt{x^2-t^2}}\ dx \ll_A (\log(3T)) \left(1+\frac{t}{T}\right)^{-A}.
\end{displaymath}
\end{lemma}

\begin{proof}
Applying the crude bound (cf.~\cite[Prop.~9]{HM})
\begin{displaymath}
W_{\nu}(x)\ll e^{-\pi x},\qquad x>T,
\end{displaymath}
we obtain
\begin{displaymath}
\int_{\max(t,T)}^\infty \frac{W_{\nu}^2(x)}{\sqrt{x^2-t^2}}\ dx \ll
\frac{1}{\sqrt{t}} \int_t^\infty \frac{e^{-2\pi x}}{\sqrt{x-t}}\ dx \ll e^{-2\pi t},
\end{displaymath}
hence it suffices to prove that
\begin{displaymath}
\int_t^{T} \frac{W_{\nu}^2(x)}{\sqrt{x^2-t^2}}\ dx \ll \log(3T),\qquad 0<t<T.
\end{displaymath}
Observing also that (cf.\ \cite[6.576.4]{GR})
\begin{displaymath}
\int_{2t}^{\infty} \frac{W_{\nu}^2(x)}{\sqrt{x^2-t^2}}\ dx \ll \int_0^{\infty} \frac{W_{\nu}^2(x)}{x}\ dx = \frac{1}{8},
\end{displaymath}
we are left with proving
\begin{displaymath}
\int_t^{2t} \frac{W_{\nu}^2(x)}{\sqrt{x^2-t^2}}\ dx \ll \log(3T),\qquad 0<t<T.
\end{displaymath}

To reformulate, we need to show
\begin{displaymath}
\int_t^{2t} \frac{W_{\nu}^2(x)}{\sqrt{x-t}}\ dx \ll t^{1/2}\log(3T),\qquad 0<t<T,
\end{displaymath}
and we shall accomplish this by using the bound (cf.~\cite[p.~679]{BHo} and \cite[Prop.~9]{HM})
\begin{displaymath}
W_{\nu}(x)\ll \min\bigl(T^{1/6},T^{1/4}|2\pi x-T|^{-1/4}\bigr),\qquad 0<x<2T.
\end{displaymath}
For $0<t\leq T/20$ this bound readily yields that
\begin{displaymath}
\int_t^{2t} \frac{W_{\nu}^2(x)}{\sqrt{x-t}}\ dx \ll \int_t^{2t} \frac{1}{\sqrt{x-t}}\ dx \ll t^{1/2},
\end{displaymath}
hence we can assume that $T/20<t<T$. Then, our task reduces to showing that
\begin{displaymath}
\int_I \frac{T^{1/3}}{\sqrt{x-t}}\ dx + \sum_M \int_{I(M)} \frac{T^{1/2}M^{-1/2}}{\sqrt{x-t}}\ dx \ll T^{1/2}\log(3T),
\end{displaymath}
where
\[I:=\{x\in[t,2t]:\ |2\pi x-T|\leq T^{1/3}\},\qquad I(M):=\{x\in[t,2t]:\ M\leq |2\pi x-T|<2M\},\]
and $M$ runs through the numbers $M\in[T^{1/3},20T]$ of the form $M=2^kT^{1/3}$ with $k\in\NN$.
Note that $I\subset[t,2t]$ is either empty or an interval of length less than $T^{1/3}$, while $I(M)\subset[t,2t]$ is a disjoint union of at most two intervals of lengths less than $M$. Therefore,
\[\int_I \frac{T^{1/3}}{\sqrt{x-t}}\ dx \ll T^{1/3}T^{1/6}=T^{1/2},\qquad\int_{I(M)} \frac{T^{1/2}M^{-1/2}}{\sqrt{x-t}}\ dx \ll T^{1/2}M^{-1/2}M^{1/2}=T^{1/2},\]
and we are done upon noting that the number of dyadic parameters $M$ is $\ll\log(3T)$.
\end{proof}

\begin{lemma}\label{lemma:whittaker_estimate}
For any $y_1,y_2,A>0$  we have, under the assumption \eqref{eq:nu-sum} and with the notation \eqref{T0},
\begin{displaymath}
\tilde{W}_{\nu_1,\nu_2}(y_1,y_2)\ll_A
(\log T_0)\sqrt{y_1y_2}\left(1+\frac{y_1}{T_0}\right)^{-A}\left(1+\frac{y_2}{T_0}\right)^{-A}.
\end{displaymath}
\end{lemma}

\begin{proof} From Stirling's approximation and \eqref{eq:nu-sum}, it follows that
\[\prod_{j=0}^2 \left| \Gamma\left(\frac{1}{2} + \frac{3}{2}\nu_j\right)\right|^{-1} \asymp \left| \Gamma\left(\frac{1}{2}+\nu\right)\right|^{-2}
\qquad\text{with}\qquad \nu:=\frac{3}{2}\nu_0.\]
Then we see from \eqref{eq:whittaker_diagonal} and the Cauchy--Schwarz inequality that
\begin{align*}
\tilde{W}_{\nu_1,\nu_2}(y_1,y_2)
&\ll\sqrt{y_1y_2}\int_0^{\infty}\frac{\left|W_\nu(y_1\sqrt{1+u})W_\nu(y_2\sqrt{1+u^{-1}})\right|}
{(1+u)^{1/4}(1+u^{-1})^{1/4}}\ \frac{du}{u}\\
&\ll\sqrt{y_1y_2}\left(\int_0^{\infty}\frac{W_\nu^2(y_1\sqrt{1+u})}{\sqrt{1+u^{-1}}}\ \frac{du}{u}\right)^{1/2}
\left(\int_0^{\infty}\frac{W_\nu^2(y_2\sqrt{1+u^{-1}})}{\sqrt{1+u}}\ \frac{du}{u}\right)^{1/2}.
\end{align*}
Applying the change of variables $x:=y_1\sqrt{1+u}$ (resp. $x:=y_2\sqrt{1+u^{-1}}$) in the last two integrals, we get
\[\tilde{W}_{\nu_1,\nu_2}(y_1,y_2)\ll\sqrt{y_1y_2}
\left(\int_{y_1}^{\infty} \frac{W_\nu^2(x)}{\sqrt{x^2-y_1^2}}\ dx\right)^{1/2}
\left(\int_{y_2}^{\infty} \frac{W_\nu^2(x)}{\sqrt{x^2-y_2^2}}\ dx\right)^{1/2}.\]
Finally, estimating these two integrals by Lemma~\ref{lemma:GL(2)_whittaker_integral}, we are done.
\end{proof}

\section{Proof of Theorem~\ref{theorem_2}}

Let $\phi$ be an arithmetically normalized Hecke--Maa{\ss} cusp form on $X$ as in Theorem~\ref{theorem_2}. Then the Fourier coefficients
\[\lambda_{\phi}(m_1,m_2)=\lambda_{\phi}(m_1,|m_2|)\]
in \eqref{eq:f-w-expansion} are actual eigenvalues of various Hecke operators (cf.\ \cite[Thm.~6.4.11]{G}). Moreover,
summing over the representatives $\gamma=\bigl(\begin{smallmatrix}a & b \\ c & d \end{smallmatrix}\bigr)$ of $\U_2(\ZZ)\bs \SL_2(\ZZ)$ in \eqref{eq:f-w-expansion} is the same as summing over the coprime pairs $(c,d)\in\ZZ^2$. Computing, as in \cite[(6.5.4)]{G}, the $\GL_2$ Iwasawa decomposition in the upper-left $2\times 2$ block of
\begin{displaymath}
\begin{pmatrix}
\gamma & \\ & 1
\end{pmatrix}
z \O_3(\RR) = \begin{pmatrix}
a & b & \\ c & d & \\ & & 1
\end{pmatrix}
\begin{pmatrix}
y_1y_2 & x_2y_1 & x_3 \\ & y_1 & x_1 \\ & & 1 \end{pmatrix} \O_3(\RR) = \begin{pmatrix}
\frac{y_1y_2}{|cz_2+d|} & * & * \\ & y_1|cz_2+d| & * \\ & & 1
\end{pmatrix} \O_3(\RR),
\end{displaymath}
where $z_2=x_2+iy_2\in\Hc_2$, we see from \eqref{eq:f-w-expansion}--\eqref{eq:whittaker_diagonal} that
\begin{displaymath}
\phi(z)\ll \sum_{m_1=1}^{\infty}\sum_{m_2=1}^{\infty} \frac{|\lambda_{\phi}(m_1,m_2)|}{m_1m_2}
\sum_{\substack{(c,d)\in\ZZ^2\\\gcd(c,d)=1}}\left|\tilde{W}_{\nu_1,\nu_2}(m_1y_1|cz_2+d|,m_2y_2|cz_2+d|^{-2})\right|.
\end{displaymath}
Taking any $\eps>0$, and applying the Cauchy--Schwarz inequality, we obtain
\begin{equation}\label{eq:after_CS}
\phi(z)\ll \left(\sum_{m_1=1}^{\infty}\sum_{m_2=1}^{\infty} \frac{|\lambda_{\phi}(m_1,m_2)|^2}{m_1^{2+2\eps}m_2^{1+\eps}}\right)^{1/2}
\left(\sum_{m_1=1}^{\infty}\sum_{m_2=1}^{\infty} \frac{m_1^{2\eps}}{m_2^{1-\eps}} F(m_1y_1,m_2y_2)^2 \right)^{1/2},
\end{equation}
where
\begin{equation}\label{eq:F-def}
F(s_1,s_2):=\sum_{\substack{(c,d)\in\ZZ^2\\\gcd(c,d)=1}}\left|\tilde{W}_{\nu_1,\nu_2}(s_1|cz_2+d|,s_2|cz_2+d|^{-2})\right|.
\end{equation}

It remains to estimate the two factors on the right hand side of \eqref{eq:after_CS}. The first factor (the arithmetic part) will be estimated
in Subsection~\ref{sub31}. The second factor (the analytic part) will be estimated in Subsection~\ref{sub33}.

\subsection{Estimating the arithmetic part}\label{sub31}
By \cite[Sections~6.3--6.4]{G}, the dual form
\begin{equation}\label{dualform}
\tilde\phi(z):=\phi\left((z^{-1})^t\right),\qquad z\in\Hc_3,
\end{equation}
is an arithmetically normalized Hecke--Maa{\ss} cusp form on $X=\GL_3(\ZZ)\bs\Hc_3$ with spectral parameters $(\nu_0,\nu_2,\nu_1)$ and Hecke eigenvalues
\[\lambda_{\tilde\phi}(m_1,m_2)=\lambda_\phi(m_2,m_1)=\ov{\lambda_\phi(m_1,m_2)}.\]
Introducing the Rankin--Selberg $L$-function (cf.\ \cite[Section~7.4]{G})
\[L\bigl(s,\phi\times\tilde{\phi}\bigr) = \zeta(3s) \sum_{m_1=1}^{\infty} \sum_{m_2=1}^{\infty}
\frac{|\lambda_{\phi}(m_1,m_2)|^2}{m_1^{2s}m_2^s},\qquad\Re s>1,\]
the first double sum on the right hand side of \eqref{eq:after_CS} can be estimated by the result of Brumley~\cite[Cor.~2]{Br} or Li~\cite[Thm.~2]{L}:
\begin{equation}\label{eq:arithmetic_final}
\sum_{m_1=1}^{\infty}\sum_{m_2=1}^{\infty} \frac{|\lambda_{\phi}(m_1,m_2)|^2}{m_1^{2+2\eps}m_2^{1+\eps}} = \frac{L\bigl(1+\eps,\phi\times\tilde{\phi}\bigr)}{\zeta(3+3\eps)} \ll_{\eps} T_0^{\eps}.
\end{equation}

\subsection{Estimating the analytic part}\label{sub33}
We decompose \eqref{eq:F-def} into dyadic subsums: for any $k,l\in\NN$, we define
\begin{displaymath}
D(k,l):=\left\{(c,d)\in\ZZ^2:\ \gcd(c,d)=1,\ 2^k<1+\frac{s_1|cz_2+d|}{T_0}\leq 2^{k+1},\ 2^l<1+\frac{s_2|cz_2+d|^{-2}}{T_0}\leq 2^{l+1}\right\}
\end{displaymath}
and
\begin{displaymath}
F_{k,l}(s_1,s_2):=\sum_{(c,d)\in D(k,l)} \left|\tilde{W}_{\nu_1,\nu_2}(s_1|cz_2+d|,s_2|cz_2+d|^{-2})\right|.
\end{displaymath}
Clearly,
\begin{displaymath}
F(s_1,s_2)=\sum_{k=0}^{\infty}\sum_{l=0}^{\infty} F_{k,l}(s_1,s_2),
\end{displaymath}
which implies, via the Cauchy--Schwarz inequality, that
\begin{equation}\label{eq:F_after_CS_1}
F(s_1,s_2)^2\ll_{\eps} \sum_{k=0}^{\infty}\sum_{l=0}^{\infty} 2^{\eps(k+l)}F_{k,l}(s_1,s_2)^2
\end{equation}
holds for any $\eps>0$.

We estimate $F_{k,l}(s_1,s_2)$ on the right hand side of \eqref{eq:F_after_CS_1} via Lemma~\ref{lemma:whittaker_estimate} as
\begin{displaymath}
F_{k,l}(s_1,s_2)\ll_{\eps,A} 2^{-A(k+l)} T_0^{\eps} \sqrt{s_1s_2} \sum_{(c,d)\in D(k,l)} |cz_2+d|^{-1/2}.
\end{displaymath}
From the two conditions on the range of $s_1,s_2$ in the definition of $D(k,l)$, it is immediate that $F_{k,l}(s_1,s_2)$ vanishes unless $s_1^2s_2\leq 2^{2k+l+3}T_0^3$, which means that
\begin{equation}\label{fkl1}
F_{k,l}(s_1,s_2)\ll_{\eps,A} 2^{-A(k+l)} T_0^{\eps} \sqrt{s_1s_2} \cdot 1_{s_1^2s_2\leq 2^{2k+l+3}T_0^3} \sum_{(c,d)\in D(k,l)} |cz_2+d|^{-1/2}.
\end{equation}
For the $(c,d)$-summation, we only use the estimate $|cz_2+d|\leq 2^{k+1} T_0 s_1^{-1}$, which amounts to saying that we have to sum over the lattice $\ZZ z_2+\ZZ \subset \CC$ in a given disk around the origin. The lattice has first successive minimum at least $\sqrt{3}/2$ and covolume $\Im z_2=y_2$ also at least $\sqrt{3}/2$. Applying again a dyadic decomposition and considering the pairs $(c,d)$ satisfying $2^m\leq |cz_2+d|<2^{m+1}$, we obtain by \cite[Lemma~1]{BHM}
\[\sum_{(c,d)\in D(k,l)} |cz_2+d|^{-1/2} \ll \sum_{\substack{m\in\ZZ\\1/2\leq 2^m\leq 2^{k+1}T_0 s_1^{-1}}} 2^{-m/2}(2^m+2^{2m}y_2^{-1})
\ll 2^{k/2}T_0^{1/2}s_1^{-1/2}+2^{3k/2}T_0^{3/2}s_1^{-3/2}y_2^{-1}.\]
Plugging this into \eqref{fkl1}, we conclude
\begin{displaymath}
F_{k,l}(s_1,s_2)\ll_{\eps,A} T_0^{\eps} 1_{s_1^2s_2\leq 2^{2k+l+3}T_0^3} \left( 2^{\left(1/2-A\right)k - Al} T_0^{1/2} s_2^{1/2} + 2^{\left(3/2-A\right)k - Al} T_0^{3/2}s_1^{-1}s_2^{1/2}y_2^{-1}\right).
\end{displaymath}

Using the previous bound in \eqref{eq:F_after_CS_1}, then substituting $s_1=m_1 y_1$ and $s_2=m_2 y_2$, and finally summing with respect to $m_1$ and $m_2$ as in \eqref{eq:after_CS}, we obtain
\begin{equation}\label{eq:endgame_main}
\begin{split}
&\sum_{m_1=1}^{\infty}\sum_{m_2=1}^{\infty} \frac{m_1^{2\eps}}{m_2^{1-\eps}} F(m_1y_1,m_2y_2)^2 \ll_{\eps,A} T_0^{2\eps} \sum_{k=0}^{\infty} \sum_{l=0}^{\infty} \sum_{m_1=1}^{\infty} \sum_{m_2=1}^{\infty} 1_{m_1^2m_2\leq 2^{2k+l+3}T_0^3y_1^{-2}y_2^{-1}} \\ & \times \frac{m_1^{2\eps}}{m_2^{1-\eps}}\left(2^{(\eps+1-2A)k + (\eps-2A)l}T_0 m_2y_2 + 2^{(\eps+3-2A)k + (\eps-2A)l} T_0^3 m_1^{-2}m_2y_1^{-2}y_2^{-1}\right).
\end{split}
\end{equation}
The first term in the bottom line of \eqref{eq:endgame_main} contributes (after the summation in the first line there)
\begin{displaymath}
\ll_{\eps,A} T_0^{1+2\eps}y_2\sum_{k=0}^{\infty}\sum_{l=0}^{\infty}2^{(\eps+1-2A)k + (\eps-2A)l} \sum_{m_1=1}^{\infty} m_1^{2\eps} \sum_{1\leq m_2\leq 2^{2k+l+3}T_0^3 m_1^{-2}y_1^{-2}y_2^{-1}} m_2^{\eps}.
\end{displaymath}
Here, the $m_2$-sum is at most $(2^{2k+l+3}T_0^3 m_1^{-2}y_1^{-2}y_2^{-1})^{1+\eps}$. Then clearly the resulting $m_1$-sum is convergent, which altogether means that this first term gives rise to
\begin{equation}\label{eq:endgame_term1}
\ll_{\eps,A} T_0^{4+5\eps}y_1^{-2} \sum_{k=0}^{\infty} \sum_{l=0}^{\infty} 2^{(3\eps+3-2A)k + (2\eps+1-2A)l}.
\end{equation}
Similarly, the second term in the bottom line of \eqref{eq:endgame_main} contributes (after the summation in the first line there)
\begin{displaymath}
\ll_{\eps,A} T_0^{3+2\eps} y_1^{-2}y_2^{-1} \sum_{k=0}^{\infty}\sum_{l=0}^{\infty} 2^{(\eps+3-2A)k + (\eps-2A)l} \sum_{m_1=1}^{\infty} m_1^{2\eps-2} \sum_{1\leq m_2\leq 2^{2k+l+3}T_0^3m_1^{-2}y_1^{-2}y_2^{-1}}m_2^{\eps}.
\end{displaymath}
Here, the $m_2$-sum is again at most $(2^{2k+l+3}T_0^3 m_1^{-2}y_1^{-2}y_2^{-1})^{1+\eps}$, and the resulting $m_1$-sum is convergent as above. In the end, the second term gives rise to
\begin{equation}\label{eq:endgame_term2}
\ll_{\eps,A} T_0^{6+5\eps} y_1^{-4}y_2^{-2} \sum_{k=0}^{\infty} \sum_{l=0}^{\infty} 2^{(3\eps+5-2A)k + (2\eps+1-2A)l}.
\end{equation}

We see that in both \eqref{eq:endgame_term1} and \eqref{eq:endgame_term2}, the double sums over $k$ and $l$ are convergent upon choosing $A:=3$, say. Returning to \eqref{eq:endgame_main}, this means that
\begin{equation}\label{eq:analytic_final}
\sum_{m_1=1}^{\infty}\sum_{m_2=1}^{\infty} \frac{m_1^{2\eps}}{m_2^{1-\eps}} F(m_1y_1,m_2y_2)^2 \ll_{\eps} \frac{T_0^{4+5\eps}}{y_1^2} + \frac{T_0^{6+5\eps}}{y_1^4y_2^2}.
\end{equation}

\subsection{Conclusion} The inequalities \eqref{eq:after_CS}, \eqref{eq:arithmetic_final} and \eqref{eq:analytic_final} altogether give
\[\phi(z)\ll_{\eps} \frac{T_0^{2+3\eps}}{y_1} + \frac{T_0^{3+3\eps}}{y_1^2y_2} \ll_\eps \frac{\lambda_{\phi}^{1+2\eps}}{y_1} + \frac{\lambda_{\phi}^{3/2+2\eps}}{y_1^2y_2},\]
and we write this in the more symmetric form
\begin{equation}\label{phibound}
\phi(z)\ll_\eps y_2\left(\frac{\lambda_{\phi}^{1+\eps}}{y_1y_2} + \frac{\lambda_{\phi}^{3/2+\eps}}{(y_1y_2)^2}\right).
\end{equation}

Finally, we utilize the dual form introduced in \eqref{dualform} to show that \eqref{phibound} remains valid when we interchange $y_1$ and $y_2$ on the right hand side. Indeed, we can express $\phi(z)$ as
\[\phi(z)=\tilde\phi\left((z^{-1})^t\right)=\tilde\phi\left(hw(z^{-1})^t w\right),
\qquad h:=\begin{pmatrix}y_1y_2&&\\&y_1y_2&\\&&y_1y_2\end{pmatrix},
\qquad w:=\begin{pmatrix}&&1\\&1&\\1&&\end{pmatrix},\]
where we have the Iwasawa decomposition (see \cite[(6.3.2)]{G} for a similar computation)
\[hw(z^{-1})^t w=\begin{pmatrix}1 & -x_1 & x_1x_2-x_3 \\ & 1 & -x_2 \\ & & 1\end{pmatrix}
\begin{pmatrix}y_1y_2 & & \\ & y_2 & \\ & & 1\end{pmatrix}.\]
Applying \eqref{phibound} and noting that $\lambda_{\tilde\phi}=\lambda_\phi$ by \eqref{laplaceeigenvalue}, we infer
\[\phi(z)=\tilde\phi\left(hw(z^{-1})^t w\right)\ll_\eps y_1\left(\frac{\lambda_{\phi}^{1+\eps}}{y_1y_2} + \frac{\lambda_{\phi}^{3/2+\eps}}{(y_1y_2)^2}\right).\]
Together with \eqref{phibound}, this proves \eqref{thm2bound}. The proof of Theorem~\ref{theorem_2} is complete.

\section{Proof of Theorem~\ref{theorem_3}}

Let $\phi$ be an $L^2$-normalized  Maa{\ss} cusp form on $X$ as in Theorem~\ref{theorem_3}. Note that $\phi$ is allowed to be non-tempered at the archimedean place, hence its Langlands parameters $\mu:=(\mu_1,\mu_2,\mu_3)$ introduced in \eqref{mutriple} are not necessarily purely imaginary, but they satisfy by unitarity and the standard Jacquet--Shalika bounds \cite{JS}
\[\max\bigl(|\Re\mu_1|,|\Re\mu_2|,|\Re\mu_3|\bigr)\leq\frac{1}{2}\qquad\text{and}\qquad
\{\mu_1,\mu_2,\mu_3\}=\{-\ov{\mu_1},-\ov{\mu_2},-\ov{\mu_3}\}.\]
Note also that we can assume \eqref{xybounds} without loss of generality, because we are already assuming $y_1,y_2\geq\sqrt{3}/2$, and replacing $z$ by $\gamma z$ for an arbitrary \emph{unipotent} integral matrix $\gamma\in\U_3(\ZZ)$ leaves both sides of \eqref{thm3bound} unchanged.

We shall establish \eqref{thm3bound} with the help of Selberg's pre-trace formula.  As a preparation, we denote by
\[\mathfrak{a}:=\{\diag(\alpha_1,\alpha_2,\alpha_3)\in\M_3(\RR):\ \alpha_1+\alpha_2+\alpha_3=0\}\]
the Lie algebra of the diagonal torus of $\PGL_3(\RR)$, and by
$C:\PGL_3(\RR)\to\mathfrak{a}/S_3$ the Cartan projection induced from the Cartan decomposition for $\GL_3(\RR)$:
\[g=h k_1\exp (C(g)) k_2\qquad\text{with}\qquad h\in\Z(\RR),\ k_1,k_2\in\O_3(\RR).\]
We identify the complexified dual $\mathfrak{a}^*_{\CC}$ with the set of triples (cf.\ \eqref{mutriple})
\[(\kappa_1,\kappa_2,\kappa_3)\in\CC^3\qquad\text{satisfying}\qquad\kappa_1+\kappa_2+\kappa_3=0,\]
namely such a triple acts on $\mathfrak{a}_{\CC}:=\mathfrak{a}\otimes_\RR\CC$ by the $\CC$-linear map
\[\diag(\alpha_1,\alpha_2,\alpha_3)\mapsto \kappa_1\alpha_1+\kappa_2\alpha_2+\kappa_3\alpha_3.\]

\subsection{The pre-trace formula approach}
Following \cite[Section~2]{BM} and \cite[Sections~2 \& 6]{BP}, we can construct a smooth, bi-$\O_3(\RR)$-invariant function
$f_{\mu}:\PGL_3(\RR)\to\CC$ supported in a fixed compact subset $K$ (which is independent of $\mu\in\mathfrak{a}^*_{\CC}$) with the following properties. On the one hand, the function obeys (cf.\ \cite[(2.3)--(2.4)]{BM})
\begin{equation}\label{eq:inverse_spherical_transform_BP_bound}
f_{\mu}(g)\ll \lambda_{\phi}^{3/2}\left(1+\lambda_{\phi}^{1/2}\|C(g)\|\right)^{-1/2},\qquad g\in\PGL_3(\RR),
\end{equation}
where $\|\cdot\|$ stands for a fixed $S_3$-invariant norm on $\mathfrak{a}\cong\RR^2$. On the other hand, its spherical transform $\tilde{f}_{\mu}:\mathfrak{a}^*_{\CC}/S_3 \to\CC$, defined as in \cite[(17)~of~Section~II.3]{He} or \cite[(2.3)]{BP}, satisfies
\[\tilde{f}_{\mu}(\mu)\geq 1,\qquad \tilde{f}_{\mu}(\kappa)\geq 0\]
for all Langlands parameters $\kappa\in\mathfrak{a}^*_{\CC}$ occurring in $L^2(X)$, including possibly non-tempered parameters. Then, using positivity in Selberg's pre-trace formula (see \cite[(6.1)]{BP}), or more directly by a Mercer-type pre-trace inequality (cf.\ \cite[(3.15)]{BHMM}), we obtain
\begin{equation}\label{eq:pre-trace_inequality}
|\phi(z)|^2 \leq \sum_{\gamma\in\PGL_3(\ZZ)} f_\mu(z^{-1}\gamma z).
\end{equation}

We can now derive quickly a somewhat weaker version of \eqref{thm3bound}. By \eqref{eq:inverse_spherical_transform_BP_bound}, we clearly have
\begin{equation}\label{eq:inverse_spherical_transform_trivial_bound}
f_{\mu}(g)\ll \lambda_{\phi}^{3/2},\qquad g\in\PGL_3(\RR),
\end{equation}
whence by \eqref{eq:pre-trace_inequality},
\begin{equation}\label{eq:pre-trace_inequality_with_number_of_matrices}
|\phi(z)|^2 \ll \lambda_{\phi}^{3/2} \sum_{\substack{\gamma\in\PGL_3(\ZZ)\\\|C(z^{-1}\gamma z)\|\ll 1}} 1.
\end{equation}
Here, the uniform bound $\|C(z^{-1}\gamma z)\|\ll 1$ depends on the fixed compact subset $K$ in which $f_\mu$ is supported.
Labeling the entries of $\gamma$ as
\begin{equation}\label{eq:labeling_gamma}
\gamma = \begin{pmatrix} a & b & c\\ d & e & f\\ g & h & i\end{pmatrix},
\end{equation}
a straightforward calculation gives
\begin{equation}\label{eq:z_-1_gamma_z}
z^{-1}\gamma z= \begin{pmatrix} a - d x_2 + gx_1x_2 - gx_3 & \frac{b + (a - e + h x_1- dx_2 + gx_1x_2)x_2 - (h + gx_2)x_3}{y_2} & * \\ (d - g x_1) y_2 & e + d x_2 - (h + gx_2)x_1 & \frac{ f + e x_1 + dx_3 - (i + hx_1 + gx_3)x_1}{y_1}\\gy_1y_2 & (h + gx_2) y_1 & i + hx_1 + gx_3\end{pmatrix}
\end{equation}
with
\begin{equation}\label{eq:explicit_*}
*=\frac{c + bx_1 - f x_2 + (-e+i+ hx_1)x_1x_2 + ax_3 - (i + hx_1 + dx_2 - gx_1x_2)x_3 - gx_3^2}{y_1y_2}.
\end{equation}
Now $\|C(z^{-1}\gamma z)\|\ll 1$ combined with \eqref{xybounds} implies, in this order,
\begin{equation}\label{eq:entrybounds}
g\ll y_1^{-1}y_2^{-1},\qquad h\ll y_1^{-1},\qquad d\ll y_2^{-1},\qquad a,e,i\ll 1,\qquad f\ll y_1,\qquad b\ll y_2,\qquad c\ll y_1y_2.
\end{equation}
Therefore, in \eqref{eq:pre-trace_inequality_with_number_of_matrices}, the number of relevant matrices $\gamma$ is $\ll y_1^2y_2^2$, and we obtain readily
\begin{equation}\label{eq:conclusive_pre-trace_bound}
\phi(z) \ll \lambda_{\phi}^{3/4} y_1 y_2.
\end{equation}

\subsection{An interlude on the Cartan projection}
To improve upon the bound \eqref{eq:conclusive_pre-trace_bound}, we shall make full use of \eqref{eq:inverse_spherical_transform_BP_bound} instead of applying its weaker version \eqref{eq:inverse_spherical_transform_trivial_bound}. With this aim in mind, we prove the following estimate on the Cartan projection, which works in general for $\PGL_n(\RR)$ or $\SL_n(\RR)$.

\begin{lemma}\label{lemma:lower_bound_on_cartan_projection} Let $g\in\SL_n(\RR)$ be an upper-triangular matrix with positive diagonal entries, and let $\id\in\SL_n(\RR)$ denote the $n\times n$ identity matrix. If $g$ lies in a fixed compact subset of $\SL_n(\RR)$, then
\[\|C(g)\|\asymp\|g^tg-\id\|\asymp\|g-\id\|.\]
\end{lemma}

\begin{proof}
The first relation is standard and follows from the Cartan decomposition $g=k_1\exp(C(g))k_2$ with $k_1,k_2\in\SO_n(\RR)$. So let us focus on the second relation. As $g$ lies in a fixed compact subset, we have that $\|g^tg-\id\|\asymp\|g^t-g^{-1}\|$, so it suffices to prove that $\|g^t-g^{-1}\|\asymp\|g-\id\|$. Here, $g^t$ is lower-triangular with some positive diagonal $a$, and $g^{-1}$ is upper-triangular with diagonal $a^{-1}$. Therefore,
\begin{align*}\|g^t-g^{-1}\|&\asymp\|g^t-a\|+\|a-a^{-1}\|+\|a^{-1}-g^{-1}\|\\
&\asymp\|g^t-a\|+\|a-\id\|+\|\id-a^{-1}\|+\|a^{-1}-g^{-1}\|\\
&\asymp\|g^t-\id\|+\|\id-g^{-1}\|\\
&\asymp\|g-\id\|,
\end{align*}
where in the last step we used again that $g$ lies in a fixed compact subset. We are done.
\end{proof}

We are now ready to improve \eqref{eq:conclusive_pre-trace_bound}. In the light of \eqref{eq:entrybounds}, it is convenient to rewrite \eqref{eq:pre-trace_inequality} as
\begin{equation}\label{eq:pre-trace_inequality_cases}
|\phi(z)|^2 \leq \sum_{j=1}^4\sum_{\gamma\in M_j} f_\mu(z^{-1}\gamma z),
\end{equation}
where, in the notation of \eqref{eq:labeling_gamma},
\begin{displaymath}
\begin{split}
M_1&:=\{\gamma\in\PGL_3(\ZZ):\ g=0,\ h=0,\ d=0\},\\
M_2&:=\{\gamma\in\PGL_3(\ZZ):\ g=0,\ h\neq 0,\ d=0\},\\
M_3&:=\{\gamma\in\PGL_3(\ZZ):\ g=0,\ h=0,\ d\neq 0\},\\
M_4&:=\{\gamma\in\PGL_3(\ZZ):\ g\neq 0\ \ \text{or}\ \ hd\neq 0\}.
\end{split}
\end{displaymath}

\subsection{The contribution of $M_1$}
In this case, $g=h=d=0$ forces $a,e,i\in\{\pm 1\}$. Then \eqref{eq:z_-1_gamma_z}--\eqref{eq:explicit_*} simplify to
\begin{equation}\label{eq:z_-1_gamma_z special 1}
kz^{-1} \gamma z = \begin{pmatrix}
1 & \pm\frac{b + (\pm 1\pm 1)x_2}{y_2} & \pm\frac{c + b x_1 - f x_2 +(\pm 1\pm 1) x_1x_2 + (\pm 1\pm 1)x_3}{y_1y_2} \\
& 1 & \pm\frac{f + (\pm 1\pm 1)x_1}{y_1} \\ & & 1
\end{pmatrix},
\end{equation}
where $k$ abbreviates $\diag(a,e,c)\in\O_3(\RR)$, and all combinations of the $\pm$ signs are allowed.

For a dyadic parameter $\lambda_{\phi}^{-1/2}\ll K\ll 1$, let us examine the contribution to \eqref{eq:pre-trace_inequality_cases} of the matrices $\gamma\in M_1$ that satisfy $\|C(z^{-1}\gamma z)\|\asymp K$. Applying \eqref{eq:inverse_spherical_transform_BP_bound}, we obtain
\begin{equation}\label{eq:M1bound}
\sum_{\substack{\gamma\in M_1 \\ K \leq\|C(z^{-1}\gamma z)\|< 2K}} f_\mu(z^{-1}\gamma z) \ll \lambda_{\phi}^{5/4} K^{-1/2} \sum_{\substack{\gamma\in M_1 \\ K \leq\|C(z^{-1}\gamma z)\|< 2K}} 1.
\end{equation}
The matrix count on the right hand side is given by the number of choices for the integral triple $(f,b,c)\in\ZZ^3$ in \eqref{eq:z_-1_gamma_z special 1}. Applying Lemma~\ref{lemma:lower_bound_on_cartan_projection} in the form $\|kz^{-1} \gamma z-1_3\|\ll K$, we infer
\[\sum_{\substack{\gamma\in M_1 \\ K \leq\|C(z^{-1}\gamma z)\|< 2K}} 1\ll\#(f,b,c)\ll (1+y_1K)(1+y_2K)(1+y_1y_2K).\]
Keeping in mind that $y_1,y_2\gg 1$, we can now estimate the right hand side of \eqref{eq:M1bound} as
\[\lambda_{\phi}^{5/4}(K^{-1/2}+y_1y_2K^{1/2}+y_1y_2^2 K^{3/2}+y_1^2y_2K^{3/2}+y_1^2y_2^2K^{5/2}).\]
Summing up over dyadic parameters satisfying $\lambda_{\phi}^{-1/2}\ll K\ll 1$, and using $y_1,y_2\gg 1$ again, we get
\begin{displaymath}
\sum_{\substack{\gamma\in M_1 \\ \lambda_{\phi}^{-1/2}\leq\|C(z^{-1}\gamma z)\|\ll 1}} f_\mu(z^{-1}\gamma z) \ll  \lambda_{\phi}^{3/2} + \lambda_{\phi}^{5/4}y_1^2y_2^2.
\end{displaymath}

We estimate the remaining contribution of the matrices $\gamma\in M_1$ with $\|C(z^{-1}\gamma z)\|<\lambda_{\phi}^{-1/2}$ similarly, this time applying \eqref{eq:inverse_spherical_transform_trivial_bound}. In the end, we obtain
\begin{equation}\label{eq:M1contribution}
\sum_{\gamma\in M_1} f_\mu(z^{-1}\gamma z) \ll  \lambda_{\phi}^{3/2} + \lambda_{\phi}^{5/4}y_1^2y_2^2.
\end{equation}

\subsection{The contribution of $M_2$ and $M_3$} We work out the relevant bound only for $M_2$, as the argument for $M_3$ is very similar. As above, $g=d=0$ forces $a=\pm 1$. In this case, since $h\neq 0$, we see from \eqref{eq:entrybounds} that $y_1\ll 1$ and there are only $O(1)$ choices for the bottom right $2\times 2$ block $\bigl(\begin{smallmatrix} e & f \\ h & i\end{smallmatrix}\bigr)\in\GL_2(\ZZ)$ of $\gamma$. For each such block, we multiply $z^{-1}\gamma z$ by an orthogonal matrix $k\in\bigl(\begin{smallmatrix}a & \\ & \O_2(\RR) \end{smallmatrix}\bigr)$ on the left to arrive at an upper-triangular matrix with positive diagonal entries\footnote{In the case of $M_3$, we would multiply $z^{-1}\gamma z$
by an orthogonal matrix $k\in\bigl(\begin{smallmatrix}\O_2(\RR)&\\&i\end{smallmatrix}\bigr)$ on the right, to achieve the same.}.  This multiplies the top row by $a=\pm 1$, i.e.\ \eqref{eq:z_-1_gamma_z}--\eqref{eq:explicit_*} become
\begin{displaymath}
kz^{-1}\gamma z = \begin{pmatrix} 1 & \pm\frac{b + (\pm 1 - e + h x_1)x_2 - hx_3}{y_2} & \pm\frac{c + bx_1 - f x_2 + (-e+i+ hx_1)x_1x_2 \pm x_3 - (i + hx_1)x_3}{y_1y_2} \\ & * & * \\ & & *
\end{pmatrix}.
\end{displaymath}
For a dyadic parameter $\lambda_{\phi}^{-1/2}\ll K\ll 1$, we estimate similarly as in \eqref{eq:M1bound},
\begin{displaymath}
\sum_{\substack{\gamma\in M_2 \\ K \leq\|C(z^{-1}\gamma z)\|< 2K}} f_\mu(z^{-1}\gamma z) \ll \lambda_{\phi}^{5/4} K^{-1/2} \sum_{\substack{\gamma\in M_2 \\ K \leq\|C(z^{-1}\gamma z)\|< 2K}} 1.
\end{displaymath}
Applying Lemma~\ref{lemma:lower_bound_on_cartan_projection} again, we can bound the matrix count on the right hand side as
\begin{displaymath}
\sum_{\substack{\gamma\in M_2 \\ K \leq\|C(z^{-1}\gamma z)\|< 2K}} 1\ll\#(b,c)\ll (1+y_2K)(1+y_1y_2K)\ll 1+y_2^2 K^2,
\end{displaymath}
and a similar calculation as before leads to the bound
\begin{equation}\label{eq:M2contribution}
\sum_{\gamma\in M_2} f_\mu(z^{-1}\gamma z) \ll  \lambda_{\phi}^{3/2} + \lambda_{\phi}^{5/4}y_2^2.
\end{equation}
The analogous argument applied to $M_3$ yields
\begin{equation}\label{eq:M3contribution}
\sum_{\gamma\in M_3} f_\mu(z^{-1}\gamma z) \ll  \lambda_{\phi}^{3/2} + \lambda_{\phi}^{5/4}y_1^2.
\end{equation}

\subsection{The contribution of $M_4$} We see from the first three entries of \eqref{eq:entrybounds} that $M_4$ can only contribute when $y_1\ll 1$ and $y_2\ll 1$, and then from \eqref{eq:entrybounds} and \eqref{eq:inverse_spherical_transform_trivial_bound} we obtain readily that
\begin{equation}\label{eq:M4contribution}
\sum_{\gamma\in M_4} f_\mu(z^{-1}\gamma z) \ll  \lambda_{\phi}^{3/2}.
\end{equation}

\subsection{Conclusion} By \eqref{eq:pre-trace_inequality_cases}--\eqref{eq:M4contribution}, we arrive at \eqref{thm3bound}, and the proof of Theorem~\ref{theorem_3} is complete.

\section{Proof of Theorem~\ref{theorem_1}}

Let $\phi$ be an arithmetically normalized Hecke--Maa{\ss} cusp form on $X$ as in Theorem~\ref{theorem_1}. Then \eqref{thm2bound}, and also \eqref{thm3bound} with an extra factor of $\lambda_\phi^\eps$, hold in a suitable fundamental domain for $X=\GL_3(\ZZ)\bs\Hc_3$ in the light of \eqref{xybounds}, Theorems~\ref{theorem_2}--\ref{theorem_3}, and the remarks made after Theorem~\ref{theorem_1}. We balance between these two bounds as follows.

If $y_1y_2>\lambda_{\phi}^{7/20}$, then we apply \eqref{thm2bound} to obtain
\begin{displaymath}
\phi(z) \ll_\eps\min(y_1,y_2)\left(\frac{\lambda_{\phi}^{1+\eps}}{y_1y_2} + \frac{\lambda_{\phi}^{3/2+\eps}}{(y_1y_2)^2}\right) \leq \frac{\lambda_{\phi}^{1+\eps}}{(y_1y_2)^{1/2}} + \frac{\lambda_{\phi}^{3/2+\eps}}{(y_1y_2)^{3/2}} \ll \lambda_{\phi}^{39/40+\eps}.
\end{displaymath}
Else, if $y_1y_2\leq \lambda_{\phi}^{7/20}$, we apply \eqref{thm3bound} with an extra factor of $\lambda_\phi^\eps$ to obtain
\begin{displaymath}
\phi(z)\ll_\eps\lambda_{\phi}^{3/4+\eps} + \lambda_{\phi}^{5/8+\eps}y_1y_2 \ll \lambda_{\phi}^{39/40+\eps}.
\end{displaymath}
The proof of Theorem~\ref{theorem_1} is complete.

\end{document}